\documentclass[a4paper,12pt]{article}
\usepackage{amsmath}
\usepackage{amsthm}
\usepackage{amssymb}
\usepackage{enumerate}

\title{A connection of the Brascamp-Lieb inequality with Skorokhod embedding}

\author{Yuu Hariya\thanks{Mathematical Institute, 
Tohoku University, Aoba-ku, Sendai 980-8578, Japan. }}

\date{\empty}
\numberwithin{equation}{section}

\theoremstyle{plain}

\newtheorem{thm}{Theorem}[section]

\newtheorem{prop}{Proposition}[section]

\newtheorem{lem}{Lemma}[section]

\theoremstyle{definition}

\theoremstyle{remark}
\newtheorem{rem}{Remark}[section]
\newtheorem{exm}{Example}[section]

\begin{document}

%%%%%% Symbols %%%%%%%
\def\N {\mathbb{N}}
\def\R {\mathbb{R}}
\def\Q {\mathbb{Q}}

\def\calF {\mathcal{F}}

\def\G {G}
\def\X {X}
\def\Y {Y}
\def\O {\Omega}
\def\V {V}
\def\U {U}
\def\tV {\widetilde{\V }}
\def\fm {F_{\mu }}
\def\tfm {\widetilde{F}_{\mu }}
\def\bfinv {\b \circ \fm ^{-1}}
\def\k {k}

\def\kp {\kappa}

\def\ind {\boldsymbol{1}}

\def\al {\alpha }
\def\la {\lambda }

\def\ga {\gamma }

\def\br {B}

\def\W {W}

\def\pdpinv {\Phi '\circ \Phi ^{-1}}
\def\fdfinv {\fm '\circ \fm ^{-1}}
\def\b {b}

\def\A {\mathsf{a}}
\def\v {v}
\def\p {\mathsf{p}}

%%%%%%% Measures %%%%%%%
\def\pr {P}
\def\pbes {P^{(3)}}
\def\P {\mathbb{P}}

%%%%%%% Expectations %%%%%%%
\def\ex {E}
\def\ebes {E^{(3)}}
\def\E {\mathbb{E}}

\newcommand\ND{\newcommand}
\newcommand\RD{\renewcommand}

\ND\lref[1]{Lemma~\ref{#1}}
\ND\tref[1]{Theorem~\ref{#1}}
\ND\pref[1]{Proposition~\ref{#1}}
\ND\sref[1]{Section~\ref{#1}}
\ND\ssref[1]{Subsection~\ref{#1}}
\ND\aref[1]{Appendix~\ref{#1}}
\ND\rref[1]{Remark~\ref{#1}} 
\ND\cref[1]{Corollary~\ref{#1}}
\ND\eref[1]{Example~\ref{#1}}
\ND\fref[1]{Fig.\ {#1} }
\ND\lsref[1]{Lemmas~\ref{#1}}
\ND\tsref[1]{Theorems~\ref{#1}}
\ND\dref[1]{Definition~\ref{#1}}
\ND\psref[1]{Propositions~\ref{#1}}

\ND\var[1]{\mathrm{var}(#1)}
\ND\inner[1]{\v \cdot #1}

\def\thefootnote{{}}

\maketitle 

\begin{abstract}
We reveal a connection of the Brascamp-Lieb inequality with Skorokhod 
embedding. Error bounds for the inequality in terms of variance are also 
provided. \footnote{E-mail: hariya@math.tohoku.ac.jp}
\end{abstract}

%%%%%% New section %%%%%%
\section{Introduction}\label{;intro}
The Brascamp-Lieb moment inequality plays an important role in statistical mechanics, such as in the analysis of gradient interface models; see, e.g., 
\cite{fs,dgi,gos}. 
It asserts that 
centered moments of a distribution with log-concave density relative to a 
Gaussian distribution do not exceed those of that Gaussian's; it is used to 
derive the tightness of finite-volume Gibbs measures describing the static 
interface, strict convexity of the associated surface tension, and so on. 
%%(see, e.g., \cite{fs,dgi,gos}). 

The Skorokhod embedding problem is to find a stopping time $T$ for 
%%a given 
one-dimensional Brownian motion $B$ such that $B(T)$ is distributed 
as a given probability measure on $\R $. The problem was proposed by 
Skorokhod \cite{sk} and a number of solutions have been constructed 
since then (\cite{ob}); they are applied to the proof of Donsker's invariance principle, 
robust pricings of options in mathematical finance (see, e.g., \cite{ho}), 
and so on. 

In this paper, we reveal a connection between the Brascamp-Lieb inequality 
and the Skorokhod embedding of Bass \cite{ba}; as a by-product, we also 
provide error bounds for the inequality in terms of variance by applying the 
It\^o-Tanaka formula. Let $\Y $ be an $n$-dimensional Gaussian random variable defined on a probability space $(\O ,\calF, \pr )$ with law $\nu $. 
Let $\X $ be an $n$-dimensional random variable on $(\O ,\calF ,\pr )$, whose 
law $\mu $
is given in the form 
\begin{align}\label{;density}
 \mu (dx)=\frac{1}{Z}e^{-\V (x)}\,\nu (dx)
\end{align}
with $\V $ a convex function on $\R ^{n}$ such that 
\begin{align*}
Z:=\int _{\R ^{n}}e^{-\V (x)}\,\nu (dx)<\infty . 
\end{align*}
In what follows, we fix $v\in \R ^{n}$ ($\v \neq 0$) arbitrarily. 
For a one-dimensional random variable $\xi $, we denote its variance by $\var{\xi }$: $\var{\xi }=\ex [(\xi -\ex [\xi ])^2]$. We set $\A :=\var{\v \cdot Y}$. Here $a\cdot b$ denotes the inner 
product of $a,b\in \R ^{n}$. We also set 
\begin{align*}
 \p (t;x):=\frac{1}{\sqrt{2\pi t}}\exp \left( -\frac{x^{2}}{2t}\right) , 
 \quad t>0, \ x\in \R . 
\end{align*}
The result of this paper is stated as follows. 

\begin{thm}\label{;tBL} 
For every convex function $\psi $ on $\R $, 
we have the following. 
\begin{enumerate}[(i)]{}
\item It holds that 
 \begin{align}\label{;BL1}
  \ex \left[ 
  \psi \left( \inner{Y}-\ex \left[ \inner{Y}\right] \right) 
  \right] 
  \ge 
  \ex \left[ 
  \psi \left( \inner{X}-\ex \left[ \inner{X}\right] \right) 
  \right] . 
 \end{align}
 More precisely, we have 
 \begin{align}\label{;BL2}
  \ex \left[ 
  \psi \left( \inner{Y}-\ex \left[ \inner{Y}\right] \right) 
  \right] 
  \ge 
  &\ex \left[ 
  \psi \left( \inner{X}-\ex \left[ \inner{X}\right] \right) 
  \right] \notag \\
  &+\frac{1}{2}\int _{\R }\psi ''(dx)\,
  \int _{0}^{\A ^{-1}\left( \A -\var{\inner{X}}\right) ^{2}}
  \!\!\!ds\,\p \left( s;\sqrt{x^{2}+\A }\right) , 
 \end{align}
 where $\psi ''(dx)$ denotes the second derivative of $\psi $ in the sense of 
 distribution. 

\item For every $p>1$, it holds that 
 \begin{align}\label{;BL3}
  \ex \left[ 
  \psi \left( \inner{Y}-\ex \left[ \inner{Y}\right] \right) 
  \right] 
  \le  
  &\ex \left[ 
  \psi \left( \inner{X}-\ex \left[ \inner{X}\right] \right) 
  \right] \notag \\
  &+C(\A ,\psi ,q)\left( \A -\var{\inner{X}}\right) ^{\frac{1}{2p}}. 
 \end{align}
 Here $C(\A ,\psi ,q)\in [0,\infty ]$ is given by 
 \begin{align*}
  C(\A ,\psi ,q)=\left( 
  \A (1+q)
  \right) ^{\frac{1}{2q}}\int _{\R }\psi ''(dx)\,
  \p \biggl( 1;\frac{x}{\sqrt{\A (1+q)}}\biggr) 
 \end{align*}
 with $q$ the conjugate of $p$: $p^{-1}+q^{-1}=1$. Note that 
 $\A -\var{\inner{X}}\ge 0$ by \eqref{;BL1}. 
\end{enumerate}
The above inequalities \eqref{;BL1}--\eqref{;BL3} are understood 
to hold also in the case that both sides of them are infinity. 
\end{thm}

\begin{rem}
\thetag{1}~The inequality \eqref{;BL1} is called the Brascamp-Lieb 
inequality. It was originally proved by Brascamp and Lieb 
\cite[Theorem~5.1]{bl} in the case $\psi (x)=|x|^{p},\,p\ge 1$; it was then  
extended to general convex $\psi $'s by Caffarelli 
\cite[Corollary~6]{ca} based on a deep understanding of 
optimal transportation between $\mu $ and $\nu $, and the related 
Monge-Amp\'ere equation. \\
\thetag{2}~In the case $\psi ''(\R )<\infty $, letting $p\to 1$ 
in \eqref{;BL3} yields 
\begin{align*}
 \ex \left[ 
  \psi \left( \inner{Y}-\ex \left[ \inner{Y}\right] \right) 
  \right] 
  -
  \ex \left[ 
  \psi \left( \inner{X}-\ex \left[ \inner{X}\right] \right) 
  \right] \le \frac{1}{\sqrt{2\pi }}\psi ''(\R )
  \left( \A -\var{\inner{X}}\right) ^{\frac{1}{2}}. 
\end{align*}
Taking $\psi (x)=|x|$ and some manipulation show that 
\begin{align*}
 \frac{\ex \left[ \left| \inner{X}-\ex \left[ \inner{X}\right] \right| \right] }{\var{\inner{X}}}\ge \frac{1}{\sqrt{2\pi \var{\inner{Y}}}}
\end{align*}
for any convex $\V $. 
\end{rem}
\smallskip 

The rest of the paper is organized as follows: 
In \sref{;prtBL} we prove \tref{;tBL}. 
The Brascamp-Lieb inequality \eqref{;BL1} is proved in \ssref{;sspr1}; 
we devote \ssref{;sspr2} to the proof of \eqref{;BL2} and \eqref{;BL3}; 
in \ssref{;sspr3} we prove \lref{;gd}, which plays an essential role 
in the proof of \tref{;tBL}. In the appendix we discuss an 
extension of the Brascamp-Lieb inequality to the case with $\V $ 
not necessarily convex. 
\smallskip 

For every function $f$ on $\R $ and $x\in \R $, we denote respectively by 
$f'_{+}(x)$ and $f'_{-}(x)$ the right- and left-derivatives of $f$ at $x$ if 
they exist. For each $x,y\in \R $, we write $x\wedge y=\min \{ x,y\} $ and 
$x^{+}=\max \{ x,0\} $. Other notation will be introduced as 
needed. 

\section{Proof of \tref{;tBL}}\label{;prtBL}
In this section we give a proof of \tref{;tBL}. Without loss of 
generality, we may assume that $\nu $ is centered: $\ex [Y]=0$. 
Moreover, Theorem~4.3 of \cite{bl} reduces the proof to the case $n=1$; 
that is, the density of the law $\pr \circ \left( v\cdot \X \right) ^{-1}$ 
relative to the one-dimensional Gaussian measure 
$\pr \circ \left( v\cdot \Y \right) ^{-1}$ is log-concave. 
Therefore in what follows, we take the Gaussian 
measure $\nu $ in \eqref{;density} as 
\begin{align*}
 \nu (dx)=\frac{1}{\sqrt{2\pi \A }}\exp \left( 
 -\frac{x^2}{2\A } 
 \right) dx, \quad x\in \R , 
\end{align*}
and $V$ as a convex function on $\R $.  We accordingly write 
$X$ and $Y$ for $\inner{X}$ and $\inner{Y}$, respectively; 
that is, $X$ is distributed as $\mu $ and $Y$ as $\nu $. 

\subsection{Proof of \eqref{;BL1}}\label{;sspr1}
In this subsection we prove the inequality \eqref{;BL1} in 
\tref{;tBL}. We denote by $\fm $ the distribution function of $\mu $: 
\begin{align*}
 \fm (x):=\frac{1}{Z}\int _{-\infty }^{x}e^{-\V (y)}\,\nu (dy), 
 \quad x\in \R . 
\end{align*}
We also set 
\begin{align*}
 \Phi (x):=\frac{1}{\sqrt{2\pi }}\int _{-\infty }^{x}\exp \left( 
 -\frac{1}{2}y^2
 \right) dy, \quad x\in \R , 
\end{align*}
and 
\begin{align}\label{;defg}
 g:=\fm ^{-1}\circ \Phi . 
\end{align}
Here $\fm ^{-1}:(0,1)\to \R $ is the inverse function of $\fm $. 
Apparently $g$ is strictly increasing. By convexity of $V$ we have moreover 
\begin{lem}\label{;gd}
It holds that $g'(x)\le \sqrt{\A }$ for all $x\in \R $. 
\end{lem}

We postpone the proof of this lemma to \ssref{;sspr3}. 
Once this lemma is shown, the proof of \eqref{;BL1} is straightforward from 
the Skorokhod embedding of Bass \cite{ba}; for other types of embeddings, 
we refer the reader to the detailed survey \cite{ob} by Ob\l \'oj. Let 
$\{ \W _{t}\} _{t\ge 0}$ be a standard one-dimensional Brownian motion 
on $(\O ,\calF ,\pr )$. 
\begin{proof}[Proof of \eqref{;BL1}]
Note that $g(\W _{1})$ is distributed as $\mu $. Applying 
Clark's formula to $g(\W _{1})$ yields 
\begin{align*}
 g(\W _{1})-\ex \left[ g(\W _{1})\right] 
 =\int _{0}^{1}a(s,\W _{s})\,d\W _{s} \quad \text{$\pr $-a.s.}, 
\end{align*}
where for $0\le s\le 1$ and $y\in \R $, 
\begin{align}\label{;arepr}
 a(s,y):=&\frac{\partial }{\partial y}\ex \left[ g(y+\W _{1-s})
 \right] \notag \\
 =&\ex \left[ 
 g'(y+\W _{1-s})
 \right] . 
\end{align}
By the Dambis-Dubins-Schwarz theorem (see, e.g., \cite[Theorem~V.1.6]{ry}), 
there exists a Brownian motion $\{ \br (t)\} _{t\ge 0}$ 
on $(\O ,\calF ,\pr )$ such that 
\begin{align*}
 \int _{0}^{t}a(s,\W _{s})\,d\W _{s}=\br \left( 
 \int _{0}^{t}a(s,\W _{s})^2\,ds
 \right) \quad \text{for all $0\le t\le 1$, $\pr $-a.s.}
\end{align*}
We know from \cite{ba} that $T:=\int _{0}^{1}a(s,\W _{s})^2\,ds$ is 
a stopping time in the natural filtration of $\br $. Moreover, 
by \eqref{;arepr} and \lref{;gd}, we have 
$
 T\le \A 
$ 
$\pr $-a.s. We denote by $\{ L^{x}_{t}\} _{t\ge 0,x\in \R }$ 
the local time process of $B$. For every $x\in \R $, Tanaka's formula yields 
\begin{align} 
 \ex \left[ 
 \left( \br (\A )-x\right) ^{+}
 \right] 
 &=\ex \left[ 
 \left( \br (T)-x\right) ^{+}
 \right] +\frac{1}{2}\ex \left[ L^{x}_{\A }-L^{x}_{T}\right] , \label{;tf1}\\
 \ex \left[ \left( x-\br (\A )\right) ^{+}
 \right] 
 &=\ex \left[ 
 \left( x-\br (T)\right) ^{+}
 \right] +\frac{1}{2}\ex \left[ L^{x}_{\A }-L^{x}_{T}\right] . \label{;tf2}
\end{align}
From \eqref{;tf1} and \eqref{;tf2}, it follows that for every 
convex $\psi $, 
\begin{align}\label{;tf3}
 \ex \left[ 
 \psi \left( \br (\A )\right) 
 \right] 
 =\ex \left[ 
 \psi \left( \br (T)\right) 
 \right] +\frac{1}{2}\int _{\R }\psi ''(dx)\,
 \ex \left[ L^{x}_{\A }-L^{x}_{T}\right] . 
\end{align}
Indeed, by Fubini's theorem, 
\begin{align*}
 &\int _{[0,\infty )}\psi ''(dx)\,\ex \left[ 
 \left( \br (\A )-x\right) ^{+}
 \right] 
 +\int _{(-\infty ,0)}\psi ''(dx)\,
 \ex \left[ 
 \left( x-\br (\A )\right) ^{+}
 \right] \\
 &=\ex \left[ 
 \psi \left( \br (\A )\right) -\psi '_{-}(0)\br (\A )-\psi (0)
 \right] \\
 &=\ex \left[ \psi \left( \br (\A )\right)\right] -\psi (0), 
\end{align*}
which is equal, by \eqref{;tf1}, \eqref{;tf2} and 
$\ex \left[ \br (T)\right] =0$, to the right-hand side of 
\eqref{;tf3} with $\psi (0)$ subtracted. Hence \eqref{;tf3} holds. 
As $T\le \A $ a.s.\ and $\psi ''\ge 0$, it is immediate from 
\eqref{;tf3} that 
\begin{align}\label{;BLd}
 \ex \left[ \psi \left( \br (\A )\right) \right] 
 \ge \ex \left[ \psi \left( \br (T)\right) \right] , 
\end{align}
which is nothing but \eqref{;BL1} since 
\begin{align}\label{;inlaw}
 \br (T)=g(\W _{1})-\ex \left[ g(\W _{1})\right] 
 \stackrel{(d)}{=}X-\ex \left[ X\right] 
\end{align}
and $\br (\A )\stackrel{(d)}{=}Y$. The proof is complete. 
\end{proof}

\begin{rem}
\thetag{1}~For any convex $\psi $ such that 
$\int _{0}^{\cdot }\psi '_{-}\left( \br (s)\right) dB(s)$ is 
a martingale, the identity \eqref{;tf3} is immediate from 
the It\^o-Tanaka formula. \\
\thetag{2}~For any convex $\psi $ such that 
$\ex \left[ \left| \psi \left( \br (\A )\right) \right| \right] <\infty $ 
(i.e., $\ex \left[ \psi \left( \br (\A )\right) \right] 
<\infty $), the inequality \eqref{;BLd} follows readily from 
the optional sampling theorem applied to the submartingale 
$\left\{ \psi (\br (t))\right\} _{0\le t\le \A }$. 
\end{rem}

\subsection{Proof of \eqref{;BL2} and \eqref{;BL3}}\label{;sspr2}
In this subsection we prove the inequalities \eqref{;BL2} and 
\eqref{;BL3} in \tref{;tBL}. We keep the notation in the previous 
subsection. By \eqref{;tf3}, the proof is reduced to showing the 
following proposition. 

\begin{prop}\label{;estloc}
\thetag{1}~It holds that 
\begin{align}\label{;est1}
 \ex \left[ 
 L^{x}_{\A }-L^{x}_{T}
 \right] \ge 
 \int _{0}^{\A ^{-1}\left( \A -\var{X}\right) ^{2}}
  \!\!ds\,\p \left( s;\sqrt{x^{2}+\A }\right) 
  %%\quad \text{for all }x\in \R . 
\end{align}
for all $x\in \R $. \\
\thetag{2}~For every $p>1$, it holds that 
\begin{align}\label{;est2}
 \ex \left[ 
 L^{x}_{\A }-L^{x}_{T}
 \right] 
 \le 
 2\left( \A (1+q)\right) ^{\frac{1}{2q}}
 \p \biggl( 1;\frac{x}{\sqrt{\A (1+q)}}\biggr) 
  \left( \A -\var{X}\right) ^{\frac{1}{2p}}
 %%\quad \text{for all }x\in \R . 
\end{align}
for all $x\in \R $. 
\end{prop}

To prove these estimates, we prepare a lemma. 

\begin{lem}\label{;lloc}
 For every $t>0$ and $x\in \R $, we have 
 \begin{align}%%\label{;}
  \ex \left[ L^{x}_{t}\right] 
  &=\int _{0}^{t}ds\,\p (s;x) \label{;loc1} \\
  &=2\int _{0}^\infty dy \left( y-|x|\right) ^{+}\p (t;y) \label{;loc2} \\
  &=2\int _{0}^\infty dy \left( \sqrt{t}\,y-|x|\right) ^{+}\p (1;y). 
  \label{;loc3}
 \end{align}
\end{lem}

\begin{proof}
 The first equality is seen from the occupation time formula. The second 
 is due to the identity 
 \begin{align*}
  \left\{ L^{x}_{t}\right\} _{t\ge 0}
  \stackrel{(d)}{=}
  \Bigl\{ \bigl( \max _{0\le s\le t}\br (s)-|x|\bigr) ^{+} \Bigr\} _{t\ge 0}
 \end{align*}
 for every $x\in \R $, which is deduced from L\'evy's theorem 
 for Brownian local time. The third one follows from change of variables. 
\end{proof}

The proof of the proposition then proceeds as follows. Recall 
$T\le \A $ a.s. 
\begin{proof}[Proof of \pref{;estloc}]
\thetag{1}~By the strong Markov property of Brownian motion, 
\begin{align}\label{;prp1}
 \ex \left[ L^{x}_{\A }-L^{x}_{T}\right] 
 =\ex \left[ 
 \ex \left[ L^{x-z}_{\A -t}\right] \Big| _{(t,z)=(T,\br (T))}
 \right] . 
\end{align}
By \eqref{;loc3}, this is rewritten as 
\begin{align}\label{;prp2}
 2\ex \left[ 
 \int _{0}^\infty dy 
 \left( 
 \sqrt{\A -T}\,y-|x-\br (T)|
 \right) ^{+}\p (1;y)
 \right] . 
\end{align}
Using Fubini's theorem and Jensen's inequality, we bound this from below 
by 
\begin{align*}
 2\int _{0}^\infty dy \left( 
 \ex 
 %%\left[ 
 \Bigl[ 
 \sqrt{\A -T}
 \Bigr] 
 %%\right] 
 y-\ex \left[ 
 |x-\br (T)|
 \right] 
 \right) ^{+}\p (1;y). 
\end{align*}
By the optional sampling theorem and Schwarz's inequality, 
\begin{align*}
 \ex \left[ 
 |x-\br (T)|
 \right] &\le \ex \left[ |x-\br (\A )|\right] \\
 &\le \sqrt{x^2+\A }. 
\end{align*}
Plugging this and using the identity between \eqref{;loc3} and 
\eqref{;loc1} lead to 
\begin{align*}
 \ex \left[ L^{x}_{\A }-L^{x}_{T}\right] \ge 
 \int _{0}^{\ex \left[ \sqrt{\A -T}\right] ^{2}}
 \!\!ds\,\p \left( s;\sqrt{x^2+\A }\right) . 
\end{align*}
Since $\sqrt{\A -t}\ge \A ^{-1/2}(\A -t)$ for $0\le t\le \A $, 
we see that 
\begin{align*}
 \ex \left[ \sqrt{\A -T}\right] ^{2}
 &\ge \A ^{-1}\left( \A -\ex [T]\right) ^{2}\\
 &=\A ^{-1}\left( \A -\var{X}\right) ^{2}, 
\end{align*}
where the equality follows from Wald's identity 
\begin{align}\label{;wald}
 \ex [T]=\ex \left[ \br (T)^{2}\right] 
\end{align}
and from \eqref{;inlaw}. This proves \eqref{;est1}. \\
\thetag{2}~First we show that for every $t>0$ and $x\in \R $, 
\begin{align}\label{;prp3}
 \ex \left[ 
 \int _{0}^{t}ds\,\p \left( s; |x-\br (T)|\right) 
 \right] 
 \le \int _{0}^{\A +t}ds\,\p (s;x). 
\end{align}
By the identity between \eqref{;loc1} and \eqref{;loc2}, and by Fubini's 
theorem, the left-hand side is equal to 
\begin{align}\label{;prp4}
 2\int _{0}^\infty dy\,\ex \left[ \left( y-|x-\br (T)|\right) ^{+}\right] 
 \p (t;y). 
\end{align}
We note the identity 
$
\left( y-|x-z|\right) ^{+}=
(z-x+y)^{+}\wedge (x+y-z)^{+}
$ 
for $z\in \R $, to bound the expectation in the integrand from above by 
\begin{align*}
 &\ex \left[ 
 \left( 
 \br (T)-x+y
 \right) ^{+}
 \right] 
 \wedge 
 \ex \left[ 
 \left( 
 x+y-\br (T)
 \right) ^{+}
 \right] \\
 &\le \ex \left[ 
 \left( 
 \br (\A )-x+y
 \right) ^{+}
 \right] 
 \wedge 
 \ex \left[ 
 \left( 
 x+y-\br (\A )
 \right) ^{+}
 \right] \\
 &=
 \ex \left[ 
 \left( 
 \br (\A )+y-|x|
 \right) ^{+}
 \right] . 
\end{align*}
Here for the inequality, we used the optional sampling theorem; 
the equality follows from the monotonicity of 
$\ex \left[ \left( \br (\A )-x+y\right) ^{+}\right] $ in $x$ and 
the symmetry in the sense that 
$
\ex \left[ \left( \br (\A )-(-x)+y\right) ^{+}\right] 
=\ex \left[ \left( x+y-\br (\A )\right) ^{+}\right] 
$. 
Therefore \eqref{;prp4} is dominated by 
\begin{align*}
 &2\int _{0}^{\infty }dy\int _{\R }dz \left( z+y-|x|\right) ^{+}
 \p (\A ;z)\p (t;y)\\
 &= 2\int _{\R }du \left( \sqrt{\A +t}\,u-|x|\right) ^{+}
 \p (1;u)\int _{-\infty }^{\sqrt{\A ^{-1}t}\,u}dv\,\p (1;v)\\
 &\le 2\int _{0}^{\infty }du \left( \sqrt{\A +t}\,u-|x|\right) ^{+}
 \p (1;u), 
\end{align*}
where we changed variables with 
$u=\frac{z+y}{\sqrt{\A +t}}$ and $v=\frac{tz-\A y}{\sqrt{\A t(\A +t)}}$ 
for the equality. Now \eqref{;prp3} follows from the identity between 
\eqref{;loc3} and \eqref{;loc1}. 

By \eqref{;prp1}, \eqref{;loc1} and H\"older's inequality, 
\begin{align*}
 \ex \left[ L^{x}_{\A }-L^{x}_{T}\right] 
 &\le \left( \frac{1}{2\pi }\right) ^{\frac{1}{2p}}
 \ex \left[ \int _{0}^{\A -T}\frac{ds}{\sqrt{s}}\right] ^{\frac{1}{p}}
 \ex \left[ 
 \int _{0}^{\A }ds\,\p \left( s;\sqrt{q}|x-\br (T)|\right) 
 \right] ^{\frac{1}{q}}\\
 &=\left( \frac{2}{\pi }\right) ^{\frac{1}{2p}}
 q^{\frac{1}{2q}}\ex \left[ \sqrt{\A -T}\right] ^{\frac{1}{p}}
 \ex \left[ 
 \int _{0}^{\A q^{-1}}ds\,\p \left( s;|x-\br (T)|\right) 
 \right] ^{\frac{1}{q}}. 
\end{align*}
By Jensen's inequality, \eqref{;wald} and \eqref{;inlaw}, 
\begin{align*}
 \ex \left[ \sqrt{\A -T}\right] \le 
 \left( \A -\ex [T]\right) ^{\frac{1}{2}}
 =\left( \A -\var{X}\right) ^{\frac{1}{2}}. 
\end{align*}
Moreover, by \eqref{;prp3}, 
\begin{align*}
 \ex \left[ 
 \int _{0}^{\A q^{-1}}ds\,\p \left( s;|x-\br (T)|\right) 
 \right] &\le 
 \int _{0}^{\A (1+q^{-1})}ds\,\p (s;x)\\
 &\le \sqrt{\frac{2\A (1+q^{-1})}{\pi }}
 \exp \left\{ 
 -\frac{qx^{2}}{2\A (1+q)}
 \right\} . 
\end{align*}
Combining these leads to \eqref{;est2} and ends the proof of \pref{;estloc}. 
\end{proof}

\begin{proof}[Proof of \eqref{;BL2} and \eqref{;BL3}]
They are immediate from \eqref{;tf3} and \pref{;estloc}. 
\end{proof}

\subsection{Proof of \lref{;gd}}\label{;sspr3}
We conclude this section with the proof of \lref{;gd}; 
the assertion itself is nothing but that of \cite[Theorem~11]{ca}. 
Here we give a different proof. To begin with, note that we only need to 
to consider the case $\A =1$; indeed, setting 
\begin{align*}
 \tV (x):=\V \left( \sqrt{\A }x\right) , && 
 \tfm (x):=\frac{\sqrt{\A }}{Z}\int _{-\infty }^{x}
 \exp \left( 
 -\frac{1}{2}y^2-\tV (y)
 \right) dy, 
\end{align*}
we have $\fm (x)=\tfm \left( x/\sqrt{\A }\right) $, from which it follows 
that 
\begin{align*}
 \fm ^{-1}\circ \Phi (x)=\sqrt{\A }\tfm ^{-1}\circ \Phi (x). 
\end{align*}
Therefore the assertion of \lref{;gd} is equivalent to 
\begin{align*}
 \left( \tfm ^{-1}\circ \Phi \right) '\le 1. 
\end{align*}
Note that $\tV $ remains convex. 

From now on we let $\A =1$.  We start with 
\begin{lem}\label{;fd}
It holds that, for all $x\in \R $, 
\begin{align*}
 \frac{\fm '}{\fm }(x)\ge \frac{\Phi '}{\Phi }\left( x+\V '_{-}(x)\right) , 
 \qquad \fm '(x)\ge \Phi '\left( x+\V '_{-}(x)\right) . 
\end{align*}
These also hold true with $\V '_{-}$ replaced by $\V '_{+}$. 
\end{lem}
\begin{proof}
Since 
$
 \V (y)-\V (x)\ge \V '_{-}(x)(y-x)
$ 
for all $x,y\in \R $, we have 
\begin{align*}
 \frac{\fm }{\fm '}(x)&=\int _{-\infty }^{x}
 \exp \left( -\frac{1}{2}y^2-\V (y)\right) dy\times \exp 
 \left( \frac{1}{2}x^2+\V (x)\right) \\
 &\le \exp 
 \left( \frac{1}{2}x^2\right) \int _{-\infty }^{x}\exp \left\{ 
 -\frac{1}{2}y^2-\V '_{-}(x)(y-x)
 \right\} dy\\
 &=\exp \left\{ 
 \frac{1}{2}\left( 
 x+\V '_{-}(x)
 \right) ^2
 \right\} 
 \int _{-\infty }^{x+\V '_{-}(x)}\exp \left( -\frac{1}{2}y^2\right) dy, 
\end{align*}
which is nothing but the first inequality. The latter is proved similarly. 
\end{proof}

We also utilize the following: 
\begin{lem}\label{;nd}
The function $\pdpinv (\xi ),\xi \in (0,1)$, is concave and symmetric with 
respect to $\xi =1/2$, and satisfies 
$\pdpinv (0+)=\pdpinv (1-)=0$. Here $\Phi ^{-1}$ is the inverse function of 
$\Phi $. 
\end{lem}
\begin{proof}
A simple calculation shows 
\begin{align*}
\left\{ 
\pdpinv 
\right\} '=-\Phi ^{-1}. 
\end{align*}
Since $\Phi ^{-1}:(0,1)\to \R $ is increasing, the concavity follows. The 
symmetry and values at boundary are obvious. 
\end{proof}

Using the above two lemmas, we prove \lref{;gd}. 
\begin{proof}[Proof of \lref{;gd}]
 Since 
 \begin{align*}
  g'(x)=\frac{\Phi '(x)}{\fdfinv \left( \Phi (x)\right) }, 
 \end{align*}
 the assertion of the lemma with $\A =1$ is equivalent to 
 \begin{align}\label{;posi}
  \G (\xi ):=\fdfinv (\xi )-\pdpinv (\xi )\ge 0 \quad 
  \text{for all }\xi \in (0,1). 
 \end{align}
 First we show that there exists $0<\delta <1$ such that both 
 \begin{align}\label{;bdry}
  \inf _{\xi \in (0,\delta ]}\G (\xi )\ge 0 \quad \text{and} \quad 
  \inf _{\xi \in [1-\delta ,1)}\G (\xi )\ge 0 
 \end{align}
 hold. Set 
 \begin{align*}
  \b (x):=\frac{\fm '(x)}{\Phi '\left( x+\V '_{-}(x)\right) }, 
  \quad x\in \R . 
 \end{align*}
 By \lref{;fd}, we have $\b (x)\ge 1$ and 
 \begin{align}\label{;order}
  \frac{\fm (x)}{\b (x)}\le \Phi \left( x+\V '_{-}(x)\right) 
 \end{align}
 for all $x\in \R $. We take $\xi \in (0,1)$ sufficiently small so that 
 \begin{align*}
  \Phi \left( x+\V '_{-}(x)\right) \big| _{x=\fm ^{-1}(\xi )}\le \frac{1}{2}. 
 \end{align*}
 Since $\pdpinv $ is increasing on $(0,1/2]$ as seen from \lref{;nd}, 
 it then holds that 
 \begin{align*}
  \Phi '\left( x+\V '_{-}(x)\right) \big| _{x=\fm ^{-1}(\xi )}
  &=\pdpinv \left( \Phi \left( x+\V '_{-}(x)\right) \right) 
  \big| _{x=\fm ^{-1}(\xi )}\\
  &\ge \pdpinv \left( \frac{\xi }{\bfinv (\xi )}\right) 
 \end{align*}
 by \eqref{;order}. Therefore, for $\xi $ sufficiently small, 
 \begin{align*}
  \G (\xi )&=\left\{ \b (x)\Phi '\left( x+\V '_{-}(x)\right) \right\} 
  \big| _{x=\fm ^{-1}(\xi )}
  -\pdpinv (\xi )\\
  &\ge \bfinv (\xi )\pdpinv \left( \frac{\xi }{\bfinv (\xi )}\right) 
  -\pdpinv (\xi ), 
 \end{align*}
 which is nonnegative since for every fixed $c\ge 1$, we have 
 \begin{align*}
  c\pdpinv \left( \frac{\eta }{c}\right) -\pdpinv (\eta )\ge 0 
  \quad \text{for all }\eta \in (0,1)
 \end{align*}
 by \lref{;nd}. We thus obtain the former inequality in \eqref{;bdry}. 
 By considering $\V (-x)$
 and using the symmetry of $\pdpinv $, 
 we also have the latter. 
 Note that $\G $ is both right- and left-differentiable since 
 $\fm '$ is and since $\fm ^{-1}$ is monotone. Suppose now that 
 $\G $ has a local minimum at some $\xi _{0}\in (0,1)$. Then 
 $\G '_{-}(\xi _{0})\le 0$ and $\G '_{+}(\xi _{0})\ge 0$. 
 Since 
 \begin{align*}
  \G '_{\pm }(\xi )
  &=\frac{\left(\fm '\right) '_{\pm}}{\fm '}\circ \fm ^{-1}(\xi )
  +\Phi ^{-1}(\xi )\\
  &=-\left( x+\V '_{\pm }(x)\right) \big| _{x=\fm ^{-1}(\xi )}
  +\Phi ^{-1}(\xi ), 
 \end{align*}
 we have 
 \begin{align*}
  \left( x+\V '_{+}(x)\right) \big| _{x=\fm ^{-1}(\xi _{0})}\le 
  \Phi ^{-1}(\xi _{0})
  \le \left( x+\V '_{-}(x)\right) \big| _{x=\fm ^{-1}(\xi _{0})}, 
 \end{align*}
 from which it follows that 
 \begin{align*}
  \Phi ^{-1}(\xi _{0})
  =\left( x+\V '_{-}(x)\right) \big| _{x=\fm ^{-1}(\xi _{0})}. 
 \end{align*}
 Hence by \lref{;fd} 
 \begin{align*}
  \G (\xi _{0})&=\left\{ 
  \fm '(x)-\Phi '\left( x+\V '_{-}(x)\right) 
  \right\} \big| _{x=\fm ^{-1}(\xi _{0})}\ge 0 . 
 \end{align*}
 Combining this observation with \eqref{;bdry}, we conclude 
 \eqref{;posi}. This completes the proof. 
\end{proof}

%%%%%%%%% Appendix %%%%%%%%%
\appendix 
\section*{Appendix}
\renewcommand{\thesection}{A}
\setcounter{equation}{0}
\setcounter{lem}{0}
\setcounter{rem}{0}
In this appendix we discuss an extension of the Brascamp-Lieb inequality 
\eqref{;BL1} to the case with potential function $\V $ not necessarily 
convex. To avoid complexity, we restrict ourselves to one-dimension; 
generalizations to multidimension may be done by considering 
one-dimensional marginals. Recently, gradient interface models 
with nonconvex potential have been studied with great interest, see, e.g., 
\cite{bk,cdm,bs}; we expect that the result presented here 
has a contribution to that study. A type of Brascamp-Lieb 
inequalities with nonconvex potential 
is also discussed by Funaki and Toukairin \cite[Section~4]{ft} with some 
restriction on convex $\psi $. 

For a given $\al >0$, suppose that the function $\k \in C^{1}(\R )$ satisfies 
\begin{align}\label{;assumk}
 k'(x)\ge \sqrt{\al } \quad \text{for all }x\in \R . 
\end{align}
Set 
\begin{align}\label{;defU}
 \U (x)=\frac{1}{2}\left| \k (x)\right| ^2-\log \k '(x), \quad x\in \R , 
\end{align}
and let the distribution $\mu $ on $\R $ be give in the form 
\begin{align*}
 \mu (dx)=\frac{1}{Z'}e^{-\U (x)}\,dx, 
\end{align*}
where the normalizing factor $Z'=\int _{\R }e^{-\U (x)}\,dx$ is 
equal to $\sqrt{2\pi }$. Let $\X $ be a random variable 
distributed as $\mu $, and $\Y $ a centered Gaussian random variable 
with variance $1/\al $. Under the above assumption, we have 

\begin{prop}\label{;pext}
For every convex function $\psi $ on $\R $, it holds that 
\begin{align}\label{;upper}
 \ex \left[ 
 \psi \left( 
 \X -\ex [\X ]
 \right) 
 \right] 
 \le \ex \left[ 
 \psi \left( 
 \Y 
 \right) 
 \right] . 
\end{align}
\end{prop}

\begin{proof}
 Since the distribution function $\fm $ of $\mu $ is written as 
 \begin{align*}
  \fm (x)&=\int _{-\infty }^{x}\frac{1}{\sqrt{2\pi }}\k '(y)
  \exp \left\{ 
  -\frac{1}{2}\left| \k (y)\right| ^2
  \right\} dy\\
  &=\Phi \left( \k (x)\right) , 
 \end{align*}
 the function $g $ defined by \eqref{;defg} is equal to $\k ^{-1}$, the 
 inverse function of $\k $. Therefore by assumption \eqref{;assumk}, we 
 have $g '(x)\le 1/\sqrt{\al }$ for all $x\in \R $, hence the same proof 
 as that of \eqref{;BL1} applies. 
\end{proof}

\begin{rem}\label{;reverse}
\thetag{1}~\lref{;gd} indicates that, by suitably adding a 
constant, the function of the form 
\begin{align*}
 \frac{1}{2}\al x^2+\V (x), \quad x\in \R , 
\end{align*}
with $\V $ convex can be expressed as \eqref{;defU} for some 
$\k $ satisfying \eqref{;assumk}. \\
\thetag{2}~In addition to \eqref{;assumk}, if we assume that 
\begin{align*}
 \k '(x)\le \sqrt{\beta } \quad \text{for all }x\in \R , 
\end{align*}
for some $\beta > \al $, then we also have 
the reverse inequality 
\begin{align}\label{;lower}
 \ex \left[ \psi (\Y ')\right] \le 
 \ex \left[ 
 \psi \left( \X -\ex [\X ]\right) 
 \right] 
\end{align}
for every convex $\psi $. 
Here $\Y '$ is a centered Gaussian random variable with variance 
$1/\beta $. 
\end{rem}

%%%%% Definition %%%%%
\def\kn {K_{\mathrm{N}}}
\def\kd {K_{\mathrm{D}}}
%%%%%%%%%%%%%%%%%%%%%%
We conclude this paper with two examples of $\U $. 
\begin{exm}[double-well type] 
Take $\al =1$ and $\k (x)=x+x^3$. Then 
\begin{align*}
 \U (x)=\frac{1}{2}x^2+x^4+\frac{1}{2}x^6-\log \left( 1+3x^2\right) . 
\end{align*}
This potential $\U $ has a double-well near the origin. 
\end{exm}

\begin{exm}[log-mixture of centered Gaussians]
For given $p,q>0$ and $0<a<b$ such that 
\begin{align}\label{;conj}
 \frac{p}{\sqrt{a}}+\frac{q}{\sqrt{b}}=1, 
\end{align}
we take 
\begin{align*}
 \k (x)=\Phi ^{-1}\left( 
 \frac{p}{\sqrt{a}}\Phi \left( \sqrt{a}x\right) 
 +\frac{q}{\sqrt{b}}\Phi \left( \sqrt{b}x\right) 
 \right) . 
\end{align*}
Then the corresponding $\U $ is expressed as 
\begin{align}\label{;forrem}
 \U (x)=-\log \left( pe^{-\frac{1}{2}ax^2}+qe^{-\frac{1}{2}bx^2}\right) . 
\end{align}
This type of potentials is dealt with in \cite{bk,cdm,bs}. The function 
$\k $ satisfies 
\begin{align}\label{;bdk}
 p\le \k '(x)\le \sqrt{b} \quad \text{for all }x\in \R , 
\end{align}
hence we have \eqref{;upper} with $\al =p^2$ and 
\eqref{;lower} with $\beta =b$. To verify \eqref{;bdk}, 
we start with the expression 
\begin{align}\label{;exkd}
 \k '(x)=
 \frac{p\Phi '\left( \sqrt{a}x\right) +q\Phi '\left( \sqrt{b}x\right) }
 {\Phi '\circ \Phi ^{-1}\left( \frac{p}{\sqrt{a}}
 \Phi \left( \sqrt{a}x\right) 
 +\frac{q}{\sqrt{b}}
 \Phi \left( \sqrt{b}x\right) 
 \right) }. 
\end{align}
To prove the lower bound, it is sufficient to take $x\le 0$ by symmetry. Then, 
as 
\begin{align*}
 \frac{p}{\sqrt{a}}
 \Phi \left( \sqrt{a}x\right) 
 +\frac{q}{\sqrt{b}}
 \Phi \left( \sqrt{b}x\right) \le 
 \Phi \left( \sqrt{a}x\right) \le \frac{1}{2}, 
\end{align*}
the denominator of \eqref{;exkd} is dominated by 
\begin{align*}
 \Phi '\circ \Phi ^{-1}\left( 
 \Phi \left( \sqrt{a}x\right) 
 \right) =\Phi '\left( \sqrt{a}x\right) 
\end{align*}
because $\Phi '\circ \Phi ^{-1}$ is increasing on 
$(0,1/2]$ (see \lref{;nd}). Therefore 
\begin{align*}
 \k '(x)\ge p+qe^{-\frac{1}{2}(b-a)x^2} 
\end{align*}
and the lower bound follows. For the upper bound, by concavity of 
$\Phi '\circ \Phi ^{-1}$ (\lref{;nd}) and relation \eqref{;conj}, 
we apply Jensen's inequality to see that 
the denominator of \eqref{;exkd} is bounded from below by 
\begin{align*}
 \frac{p}{\sqrt{a}}\Phi '\circ \Phi ^{-1}
 \left( \Phi \left( \sqrt{a}x\right) \right) 
 +\frac{q}{\sqrt{b}}
 \Phi '\circ \Phi ^{-1}
 \left( \Phi \left( \sqrt{b}x\right) \right) 
 \ge \frac{1}{\sqrt{b}}\left\{ 
 p\Phi '\left( \sqrt{a}x\right) +
 q\Phi '\left( \sqrt{b}x\right) 
 \right\} , 
\end{align*}
from which we obtain the upper bound in \eqref{;bdk}. 
We end this example with a remark that this upper bound 
also holds true in a general situation where $\k $ is given by 
\begin{align*}
 \k (x)=\Phi ^{-1}\left( 
 \int _{0}^\infty \frac{\rho (d\kp )}{\sqrt{\kp }}
 \Phi \left( \sqrt{\kp }x\right) 
 \right) 
\end{align*}
for a positive measure $\rho $ on $(0,\infty )$ 
such that its support is included in $(0,b]$ and 
\begin{align*}
\int _{0}^\infty \frac{\rho (d\kp )}{\sqrt{\kp }}=1. 
\end{align*}
The potential $\U $ corresponding to this $\k $ is given in the form 
\begin{align*}
 \U (x)=-\log \int _{0}^\infty \rho (d\kp )\,e^{-\frac{1}{2}\kp x^2}, 
\end{align*}
which is referred to as a {\it log-mixture of centered Gaussians} 
in \cite{bs}. 
\end{exm}

\begin{rem}
 For $\U $ given by \eqref{;forrem}, a concrete calculation shows that 
 in fact \eqref{;upper} holds with $\al =a$, which gives a better bound 
 than the one discussed above because $p^2\le a$ by the relation 
 \eqref{;conj}. 
\end{rem}

%%%%%%%%% References %%%%%%%%%

\end{document}